\documentclass[english, 12pt]{article}
\usepackage[T1]{fontenc}
\usepackage{babel}
\usepackage{amsmath,amsthm}
\usepackage{amssymb}
\usepackage[all]{xy}

\newcommand{\C}{\mathbb{C}}
\newcommand{\Q}{\mathbb{Q}}
\newcommand{\Z}{\mathbb{Z}}

\newcommand{\ra}{\rightarrow}

\newtheorem{thm}{Theorem}
\newtheorem{df}[thm]{Definition}
\newtheorem{prop}[thm]{Proposition}

\newtheorem{conj}[thm]{Conjecture}

\begin{document}
\title{On the zero locus of normal functions and the \'etale Abel-Jacobi map}

\author{Fran\c{c}ois Charles}

\maketitle

\begin{abstract}
In this paper, we investigate questions of an arithmetic nature related to the Abel-Jacobi map. We give a criterion for the zero locus of a normal function to be defined over a number field, and we give some comparison theorems with the Abel-Jacobi map coming from continuous \'etale cohomology.
\end{abstract}

\section{Introduction}

Let $X\ra S$ be a family of complex smooth projective varieties over a quasi-projective base, and let $Z\hookrightarrow X$ be a flat family of cycles of pure codimension $i$ which are homologically equivalent to zero in the fibers of the family. For any point $s$ of $S$, the Abel-Jacobi map associates to the cycle $Z_s$ a point in the intermediate Jacobian $J^i(X_s)$ of $X_s$, which is a complex torus (see part 2 for details). This construction works in family, yielding a bundle of complex tori, the Jacobian fibration $J^i(X/S)$, and a normal function $\nu_Z$, which is the holomorphic section of $J^i(X/S)$ associated to $Z\hookrightarrow X$. We can attach to $\nu_Z$ an admissible variation $H$ of mixed Hodge structures on $S$, see \cite{Sainf}, fitting in an exact sequence
$$0\ra R^{2i-1}f_*\Z/(\mathrm torsion) \ra H \ra \Z$$ 
such that the zero locus of $\nu_Z$ is the locus where this exact sequence splits, that is, the projection on $S$ of the locus of Hodge classes in $H$ which map to $1$in $\Z$. In analogy with the celebrated theorem of Cattani-Deligne-Kaplan on the algebraicity of Hodge loci for variations of pure Hodge structures, Green and Griffiths have stated the following conjecture, which deals with the mixed Hodge structure appearing above.

\begin{conj}
 The zero locus of the normal function $\nu$ is algebraic.
\end{conj}

In the same way that Deligne-Cattani-Kaplan's result gives evidence for the Hodge conjecture, this would give evidence for Bloch-Beilinson-type conjectures on filtration on Chow groups, see section 2.1. Actually, it has been recently obtained independently by Brosnan-Pearlstein and M. Saito in \cite{BPg} and \cite{Saialg} that if $S$ has a smooth compactification with complement a smooth divisor, then the zero locus of $\nu$ is algebraic. It seems that Brosnan and Pearlstein have proved algebraicity in case the base is a surface, and that they expect to have a proof of conjecture 1 in full generality soon.

Assume everything is defined over a finitely generated field $k$. In line with general conjectures on algebraic cycles, one would expect that not only the zero locus of $\nu$ is algebraic, but it should be defined over $k$, hence the interest in trying to investigate number-theoretic properties of the zero locus of normal functions. In the context of pure Hodge structures, i.e. that of Deligne-Cattani-Kaplan's theorem, Voisin shows in \cite{Voabs} how this question is related to the question whether Hodge classes are absolute and gives criteria for Hodge loci to be defined over number fields.

In this paper, we want to tackle such questions and also investigate the arithmetic counterpart of normal functions, namely the \'etale Abel-Jacobi map introduced by Jannsen using continuous \'etale cohomology. We give comparison results between the \'etale Abel-Jacobi map and Griffiths' Hodge-theoretic one. Recent work around the same circle of ideas can be found in \cite{GGP}. The use of Terasoma's lemma in this context is very relevant to our work, and the results proved there are closely related to our theorem \ref{gros} (though up to torsion).

\bigskip

Let us state our main results precisely. Start with a subfield $k$ of $\C$ which is generated by a finite number of elements over $\Q$, and let $S$ be a quasi-projective variety over $k$. Let $\pi : X\ra S$ be a smooth family of projective varieties of pure dimension $n$, and let $Z\hookrightarrow S$ be a family of codimension $i$ algebraic cycles of $X$, flat over $S$. Assume that $Z$ is homologically equivalent to $0$ on the geometric fibers of $\pi$. In the paper \cite{Jac}, Jannsen defines continuous \'etale cohomology, which is \'etale $l$-adic cohomology for varieties over non-algebraically closed fields. There is a cycle map from Chow groups to continuous \'etale cohomology. For any point $s$ of $S$ with value in a finitely generated extension $K$ of $k$, let $\overline{X_s}$ be the variety $X_s\times_K \overline{K}$, and $G_K=Gal(\overline{K}/K)$ the absolute Galois group of $K$. The cycle class of $Z_s$ in the continuous \'etale cohomology of $X_s$ gives a class $aj_{\mathaccent 19 et}(Z_s)\in H^1(G_K, \tilde{H}^{2i-1}(\overline{X_s}, \hat{\Z}(i)))$, where $\tilde{H}^{2i-1}(\overline{X_s}, \hat{\Z})$ denotes the quotient of $H^{2i-1}(\overline{X_s}, \hat{\Z})$ by its torsion subgroup, $\hat{\Z}$ being the profinite completion of $\Z$. This cohomology class is obtained by applying a Hoschild-Serre spectral sequence to continuous cohomology. Proposition \ref{inv} shows that the vanishing of this class is independent of the choice of $K$, i.e., it vanishes in $H^1(G_K, \tilde{H}^{2i-1}(\overline{X_s}, \hat{\Z}(i)))$ if and only if it vanishes in $H^1(G_L, \tilde{H}^{2i-1}(\overline{X_s}, \hat{\Z}(i)))$ for any finite type extension $L$ of $K$. This observation appears in \cite{Sch} and, according to one of the referees, is due to Nori.

\bigskip

It would follow from general conjectures on algebraic cycles, whether a combination of the Hodge and Tate conjectures for open varieties or versions of the Bloch-Beilinson conjectures on filtrations on Chow groups, that the zero locus of the normal function associated to the family of cycles induced by $Z_{\C}$ on $X_{\C}\ra S_{\C}$ is precisely the vanishing set of the \'etale Abel-Jacobi map. For the latter, assuming Beilinson's conjecture on Chow groups of varieties over number fields and Lewis' Bloch-Beilinson conjecture of \cite{Le}, one would know that the kernels of both Abel-Jacobi maps are equal to the second step of the unique Bloch-Beilinson filtration on Chow groups, hence that they coincide. Unfortunately, such comparison results seem to be known only for divisors and zero-cycles, where the \'etale Abel-Jacobi map can be computed using the Kummer exact sequence for Picard or Albanese varieties. We are not aware of any result in other codimension. In this paper, we therefore try to tackle such comparison results. We don't prove the general case, but we prove results of two different flavors in the variational case. We obtain such results assuming the algebraicity of the components of the zero locus of normal functions, which does not seem to be known in full generality at the time being, although as explained before it might be obtained in the near future.

In the previous situation, consider the normal function $\nu_Z$ associated to the family $Z_{\C}$ in $X_{\C}$. Its zero locus is an analytic subvariety of $S(\C)$. Our theorems are the following, where the expression "finitely generated field" denotes a field generated by a finite number of elements over its prime subfield -- $\Q$ in our case.

\begin{thm}\label{gros}
\mbox{}
\begin{itemize}
\item[(i)]  Let $T$ be an irreducible component of the zero locus of $\nu_Z$. Assume that $T$ is algebraic and that ${R^{2i-1}\pi_{\C}}_* \C$ has no nonzero global sections over $T$. Let $k$ be a finitely generated field over which $T$ is defined. Then for all point $t$ of $T$ with value in a finitely generated field, the \'etale Abel-Jacobi invariant of $Z_t$ is zero.
\item[(ii)] Assume that for every closed point $s$ of $S$ with value in a number field, the \'etale Abel-Jacobi invariant of $Z_s$ is zero and that ${R^{2i-1}\pi_{\C}}_* \C$ has no nonzero global sections over $S_{\C}$. Then $\nu_Z$ is identically zero on $S_{\C}$.
\end{itemize}
\end{thm}

\begin{thm}\label{df}
 Let $T$ be an irreducible component of the zero locus of $\nu_Z$. Assume that $T$ is algebraic and that ${R^{2i-1}\pi_{\C}}_* \C$ has no nonzero global sections over $T$. Then $T$ is defined over a finite extension of the base field $k$ and for any automorphism $\sigma$ of $\C$ over $k$, the image of $T$ by $\sigma$ is an irreducible component of $\nu_Z$.
\end{thm}

\paragraph{Remark.} In this result, we do not assume that all the components of the zero locus of $\nu_Z$ is algebraic. Furthermore, we are considering the image of $T$ by $\sigma$ as a subscheme of $S$, and we prove that it is, as a subscheme of $S$, a component of the zero locus of the holomorphic normal function $\nu_Z$. This contrasts with the situation in \cite{Voabs}, where similar results were obtained using only the reduced structure on the subschemes considered. The main difference in our setting is that the (mixed) Hodge structures we consider have at most one nonzero Hodge class, up to multiplication by a constant.

\begin{thm}\label{inter}\mbox{}
\begin{itemize}
\item[(i)] Assume that for every closed point $s$ of $S$ with value in a number field, the \'etale Abel-Jacobi invariant of $Z_s$ is zero and that there exists a complex point $s$ of $S$ such that $\nu_Z(s)=0$. Then $\nu_Z$ is identically zero on $S_{\C}$.
\item[(ii)] Let $T$ be an irreducible component of the zero locus of $\nu_Z$. Assume that $T$ is algebraic and that there exists a point $t$ of $T$ such that $aj_{\mathaccent 19 et}(Z_t)$ is zero. Then for all point $t$ of $T$ with value in a finitely generated field, the \'etale Abel-Jacobi invariant of $Z_t$ is zero.
\end{itemize}
\end{thm}

The lack of symmetry between both Abel-Jacobi maps in our results is frustrating. On the one hand, results of Serre in \cite{Se} give an almost explicit way of computing whether an Abel-Jacobi invariant is zero. This accounts for the second part of theorem \ref{inter}, which uses much stronger number-theoretic arguments than the other results and is the main part of theorem \ref{inter}. We think that the use of results such as Serre's could be useful for more general comparison theorems. On the other hand, while the local structure of zero locus of normal functions is well understood -- it is an analytic variety, and its local description is well described, see \cite{Grif2}, \cite{Voib}, ch. 17-- ,we have very few results on the ``zero locus'' of the \'etale Abel-Jacobi map. We feel that it would be very interesting to prove an \'etale counterpart of the results of \cite{BPg} and \cite{Saialg}. 

\bigskip

In this paper, if $X$ is any variety, the cohomology groups of $X$, whether singular or \'etale, will always be considered modulo torsion, so as to avoid cumbersome notations. The same convention goes for higher direct images.

\paragraph{Acknowledgement} I would like to offer my heartfelt appreciation to Claire Voisin. It is very clear that this paper owes a great deal both to her work in \cite{Voabs} and to the many discussions we had and suggestions she made, notably for the example of theorem \ref{CY}.

\section{Preliminary results on Abel-Jacobi maps}

Let $X$ be a smooth projective variety over a field $k$ of characteristic zero, $\overline{k}$ an algebraic closure of $k$ and $G_k=Gal(\overline{k}/k)$ the absolute Galois group of $k$. In his paper \cite{Belh}, Beilinson constructs a conjectural filtration $F^{\bullet}$ on the Chow groups $CH^i(X)\otimes \Q$ of $X$ with rational coefficients. It is obtained in the following way. Let $MM(k)$ be the abelian category of mixed motives over $k$. There should exist a spectral sequence, Beilinson's spectral sequence
$$E_2^{p,q}=\mathrm{Ext}^p_{MM(k)}(\mathbf{1}, \mathfrak{h}^q(X)(i))\Rightarrow \mathrm{Hom}_{D^b(MM(k))}(\mathbf{1}, \mathfrak{h}(X)(i)[p+q])$$
where $\mathfrak{h}^q(X)$ denotes the weight q part of the image $\mathfrak{h}(X)$ of $X$ in the category of pure motives. For $p+q=2i$, the abutment of this spectral sequence should be canonically isomorphic with $CH^i(X)$, hence the filtration $F$. For weight reasons, a theorem of Deligne in \cite{Deld} would imply that this spectral sequence degenerates at $E_2\otimes\Q$. We get the formula
$$Gr^{\nu}_F CH^i(X)_{\Q}=\mathrm{Ext}^{\nu}_{MM(k)}(\mathbf{1}, \mathfrak{h}^{2i-\nu}(X)(i))\otimes \Q.$$
The existence of such a filtration is also a conjecture of Bloch and Murre.

\subsection{\'Etale cohomology}
The previous construction should have its reflection in the various usual cohomology theories. Let us first consider \'etale cohomology. In the paper \cite{Jac}, Jannsen constructs continuous \'etale cohomology groups with value in the profinite completion $\hat{\Z}$ of $\Z$ for varieties over a field. Those enjoy good functoriality properties and they are equal to the usual \'etale cohomology groups in case the base field is algebraically closed. In particular, there is a Hochschild-Serre spectral sequence
\begin{equation}\label{HS}
 E_2^{p,q}=H^p(G_k, H^q(X_{\overline{k}}, \hat{\Z}(i))\Rightarrow H^{p+q}(X, \hat{\Z}(i))
\end{equation}
as well as a cycle map
$$cl^X : CH^i(X)\ra H^{2i}(X, \hat{\Z}(i)).$$

Those are compatible with the usual cycle map $cl^{X_{\overline{k}}}$ to $H^{2i}(X_{\overline{k}}, \hat{\Z}(i))$. Let $CH^i(X)_{hom}$ be the kernel of $cl^{X_{\overline{k}}}$, that is, the part of the Chow group consisting of those cycles which are homologically equivalent to zero.

\begin{df}
 The map
$$aj_{\mathaccent 19 et} :  CH^i(X)_{hom}\ra H^1(G_k, H^{2i-1}(X_{\overline{k}}, \hat{\Z}(i))$$
induced by the spectral sequence (\ref{HS}) is called the \'etale Abel-Jacobi map.
\end{df}

This map is expected to be the image by some realization functor of the analogous map coming from Beilinson's spectral sequence. As an evidence for this, we can cite Jannsen's result in \cite {JaMS}, lemma 2.7, stating that in the case $k$ is finitely generated, and for any ``reasonable'' category of mixed motives, the filtrations on $CH^i(X)\otimes \Q$ induced by Beilinson's spectral sequence and by the Hochschild-Serre spectral sequence (\ref{HS}) coincide if $cl^X$ is injective -- which is also a conjecture of Bloch and Beilinson. More specifically, if any Bloch-Beilinson filtration (see \cite {Blalg}, \cite{JaMS}, \cite{JaBB}) exists on $CH^i(X)\otimes \Q$ and $cl^X$ is injective, then it has to be the filtration induced by (\ref{HS}).

\bigskip

Our definition of the \'etale Abel-Jacobi map may seem to be highly dependent on the base field $k$, which is not convenient since we expect that for an algebraic cycle $Z$ homologically equivalent to zero, $aj_{\mathaccent 19 et}(Z)$ should reflect geometric properties of $Z$ related to the image of $Z_{\C}$ by the Abel-Jacobi map. The following proposition shows that the vanishing of $aj_{\mathaccent 19 et}(Z)$ does not actually depend of the base field. This would be false had we considered in our definition the torsion part of the cohomology of $X$. The fact that we want the following result to hold is the reason why we have to ignore this torsion part, which is related to arithmetic properties of algebraic cycles, as opposed to their geometric properties. It has been attributed to Nori and appears in a very similar form in \cite{Sch}, lemma 1.4.

\begin{prop}\label{inv}
 Let $X$ be a smooth projective variety over a finitely generated field $k$, and let $Z\in CH_{hom}^i(X)$. Let $K$ be a field which is finitely generated over $k$. Then $aj_{\mathaccent 19 et}(Z_K)=0$ if and only if $aj_{\mathaccent 19 et}(Z)=0$.
\end{prop}

\begin{proof}
We can assume that $K$ is Galois over $k$. We have an exact sequence of groups
$$1\ra G_K \ra G_k \ra \mathrm{Gal}(K/k) \ra 1,$$
hence the following exact sequence coming from the Hochschild-Serre spectral sequence
\begin{equation}\label{HSi}
0\ra H^1(\mathrm{Gal}(K/k), V^{G_K})\ra H^1(G_k, V) \ra H^1(G_K, V)^{\mathrm{Gal}(K/k)},
\end{equation}
with $V=H^{2i-1}(\overline{X}, \hat{\Z}(i))$. The definition of the \'etale Abel-Jacobi map from a Leray spectral sequence shows that $aj_{\mathaccent 19 et}(Z_K)$ is obtained from $aj_{\mathaccent 19 et}(Z)$ by the last map in (\ref{HSi}). On the other hand, it is a consequence of the Weil conjectures that $V^{G_K}$ is zero (recall $V$ is torsion-free), which implies that the last map in (\ref{HSi}) is an injection.
\end{proof}

\bigskip

The next result is due to Jannsen in \cite{JaMMK}, being inspired by results from Carlson and Beilinson we will recall later, and gives a geometric meaning to the \'etale Abel-Jacobi map. We recall it shortly, as it is crucial to the results of our paper.

Start with $X$ as before, and let $Z$ an algebraic cycle of pure codimension $i$ in $X$. Let $|Z|$ be the support of $Z$, and $U$ be the complement of $|Z|$ in $X$. By purity, we have an exact sequence of $G_k$-modules
$$0\ra H^{2i-1}(X_{\overline{k}}, \hat{\Z}(i))\ra H^{2i-1}(U_{\overline{k}}, \hat{\Z}(i)) \ra H^{0}(|Z|_{\overline{k}}, \hat{\Z})\ra 0$$
and the class of $Z$ gives a map $\hat{\Z}\ra H^{0}(|Z|_{\overline{k}}, \hat{\Z})$. The pull-back of the previous exact sequence by this map is an exact sequence of $G_k$-modules
\begin{equation}\label{exl}
0\ra H^{2i-1}(X_{\overline{k}}, \hat{\Z}(i))\ra H_{\mathaccent 19 et} \ra \hat{\Z}\ra 0.
\end{equation}
This extension gives a class $\alpha(Z)\in H^1(G_k, H^{2i-1}(X_{\overline{k}}, \hat{\Z}(i))$.

\begin{prop}
We have
$\alpha(Z)=aj_{\mathaccent 19 et}(Z)$.
\end{prop}
Let us note that this proposition immediately carries out to the variational setting for flat families of algebraic cycles. In this case, we get an extension of locally constant $\hat{\Z}$-sheaves, that is, a product of locally constant $\Z_l$-sheaves for all $l$, over the base scheme which on every fiber is canonically isomorphic to the extension (\ref{exl}).

\subsection{Hodge theory}

The Hodge-theoretic picture is different. Indeed, the category of mixed Hodge structures has no higher extension groups as shown by Beilinson, so we cannot expect to construct directly a similar filtration on Chow groups through this means. The use of Leray spectral sequences in this setting has been studied by Nori, Saito, Green-Griffiths and others, and can be considered well-understood. Even though we cannot expect to construct a Bloch-Beilinson filtration on Chow groups using Hodge theory, at least in a naive way, we can construct a two-term filtration using Deligne cohomology. We use this approach to make the similarity with the previous discussion more obvious, but in this paper we simply use Griffiths' Abel-Jacobi map, which was defined in \cite{Grif}. Griffiths' work on normal functions is of course fundamental to our results.

\bigskip

Let us assume for this paragraph that the base field is $\C$.
Recall that we have Deligne cohomology groups $H^i_{\mathcal{D}}(X, \Z(j))$. Those are the ``absolute'' version of Hodge cohomology groups in the same way that continuous \'etale cohomology is the absolute version of \'etale cohomology over an algebraic closure of the base field. There is an exact sequence
$$0\ra J^i(X) \ra  H^{2i}_{\mathcal{D}}(X, \Z(i))\ra H^{2i}(X, \Z)(i)\ra 0$$
as well as a cycle map
$$cl^X : CH^i(X)\ra H^{2i}_{\mathcal{D}}(X, \Z(i)).$$
Those are compatible with the usual cycle map to $H^{2i}(X, \Z)(i)$.

The cohomology group $H^{2i}(X, \Z)(i)$ is, up to a Tate twist the usual singular cohomology of the complex manifold $X$ with its canonical Hodge structure, and $J^i(X)$ is Griffiths' $i$-th intermediate Jacobian, which is defined the following way.

Integration of differential forms gives a map from the singular homology group $H_{2n-2i+1}(X, \Z)$ to $F^{n-i+1}H^{2n-2i+1}(X,\C)^*$, $n$ being the dimension of $X$ and $F$ the Hodge filtration. The quotient group
$$F^{n-i+1}H^{2n-2i+1}(X,\C)^*/H_{2n-2i+1}(X, \Z)$$
is a complex torus, canonically isomorphic to
$$J^i(X):=\frac{H^{2i-1}(X, \C)}{F^iH^{2i-1}(X, \C) \oplus H^{2i-1}(X, \Z)}.$$
There is a canonical isomorphism of abelian groups between $J^i(X)$ and the extension group Ext$^1_{MHS}(\Z, H^{2i-1}(X, \Z)(i))$ in the category of mixed Hodge structures, as noted by Carlson in \cite{Carl}. One can also refer to \cite{Voib}, p.463.

\begin{df}
The map
$$aj : CH^i(X)_{hom}\ra J^i(X)$$
induced from the previous exact sequence is called the (transcendental) Abel-Jacobi map, or Griffiths' Abel-Jacobi map.
\end{df}

\bigskip

In the light of the isomorphism above, Beilinson has shown in \cite{BelH} (see also \cite{Carl} for the case of divisors on curves) the following way of computing the Abel-Jacobi map, which is very similar to its \'etale counterpart-- and has been proved earlier. Let $Z$ an algebraic cycle of pure codimension $i$ in $X$. Let $|Z|$ be the support of $Z$, and $U$ be the complement of $|Z|$ in $X$. We have an exact sequence of mixed Hodge structures
$$0\ra H^{2i-1}(X, \Z(i))\ra H^{2i-1}(U, \Z(i)) \ra H^{0}(|Z|, \Z)\ra 0$$
and the class of $Z$ gives a map $\Z\ra H^{0}(|Z|, \Z)$. The pull-back of the previous exact sequence by this map is an exact sequence of mixed Hodge structures
\begin{equation}\label{exh}
0\ra H^{2i-1}(X, \Z(i))\ra H \ra \Z\ra 0.
\end{equation}
This extension gives a class $\alpha(Z)\in \mathrm{Ext}^1_{MHS}(\Z, H^{2i-1}(X, \Z)(i))=J^i(X)$.

\begin{prop}
We have
$\alpha(Z)=aj(Z)$.
\end{prop}

The vanishing of the Abel-Jacobi map has a simple interpretation in those terms. Indeed, recall that if $H$ is a mixed Hodge structure (of weight $0$) with weight filtration $W_{\bullet}$ and Hodge filtration $F^{\bullet}$, a Hodge class of weight $k$ in $H$ is an element of $W_{2k}H\cap F^k H_{\C}$. In this terms, it is straightforward to see that the extension $(\ref{exh})$ splits if and only if there exists a Hodge class (which has to be of weight $0$) in $H$ mapping to $1$ in $\Z$.

Again, in case of a flat family of algebraic cycles which are homologous to zero in the fibers, we get an extension of variations of mixed Hodge structures corresponding point by point to (\ref{exh}). It satisfies Griffiths' transversality, see \cite{Sainf}, lemma 1.3. We also get the Jacobian fibration $J^i(X/S)$, and a section $\nu_Z$ of it is obtained by applying the relative Abel-Jacobi map. The preceding remark shows that the zero locus of $\nu_Z$ is a Hodge locus for the variation of mixed Hodge structures above.

\begin{df}
 The section of $J^i(X/S)$ attached to the cycle $Z$ is the normal function $\nu_Z$ attached to $Z$.
\end{df}

Normal functions have been extensively studied, see \cite{Grif}, \cite{Sainf}, etc. See also \cite{Voib}, ch. 19. It is a fundamental fact that normal functions are holomorphic. In particular, their zero locus is analytic. In this paper, our results will take into account this analytic structure, our proofs being valid even without considering the reduced structure on the zero locus of a normal function.

\bigskip

It will be very important to us, though straightforward, that if we start with a family over a finitely generated base field, the extension of local systems coming from the \'etale Abel-Jacobi map and from the transcendental one, after base change to $\C$, are compatible in the obvious way, as Artin's comparison theorem between \'etale and singular cohomology readily shows.

\subsection{A first comparison result : Hodge classes of normal functions and their \'etale counterpart}

Let $\mathcal{H}$ be a variation of pure Hodge structures of weight $-1$ over an irreducible complex variety $S$, and let $\nu$ be a normal function on $S$. The Hodge class of $\nu$ is defined the following way. Let $H_{\Z}$ be the integral structure of $\mathcal{H}$. We have an exact sequence of sheaves on $S$
$$0\ra H_{\Z}\ra \mathcal{H}/F^0\mathcal{H}\ra \mathcal{J}(\mathcal{H})\ra 0, $$
$\mathcal{J}(\mathcal{H})$ being the sheaf of holomorphic sections of the Jacobian fibration. This gives a map $H^0(S, J(\mathcal{H}))\ra H^1(S, H_{\Z})$. The normal function $\nu$ is a holomorphic section of $\mathcal{J}(\mathcal{H})$. Its image in $H^1(S, H_{\Z})$ is called its Hodge class and is denoted by $[\nu]$.

The homological interpretation of intermediate Jacobians suggests another way of defining Hodge classes of normal functions. Indeed, a normal function $\nu$ defines an extension of variations of mixed Hodge structures over $S$
$$0\ra H\ra H' \ra \Z \ra 0.$$

The long exact sequence of sheaf cohomology gives a map $\delta : H^0(S, \Z)\ra H^1(S, \mathcal{H})$. The following is straightforward, but we have been unable to find a reference.

\begin{prop}
We have $[\nu]=\delta(1)$.
\end{prop}

\begin{proof}
Let us start by briefly recalling the explicit form given in \cite{Carl} of the isomorphism $\phi : \mathrm{Ext}^1_{MHS}(\Z, H)\simeq J(H)$ when $S$ is a point. Choose an isomorphism of abelian groups
$$H'_{\Z}\simeq H_{\Z}\oplus \Z.$$
There exists an element $\alpha\in H_{\C}$ such that $\alpha\oplus 1\in F^0 H'_{\C}$. The class of $\alpha$ in $\frac{H_{\C}}{H_{\Z}\oplus F^0H_{\C}}=J(H)$ is well-defined and is the image of the extension
$$0\ra H \ra H' \ra \Z \ra 0$$
by $\phi$.

\smallbreak

Now let us work over a general complex quasi-projective base $S$ as before. Let us choose a covering of $S(\C)$ by open subsets $U_i$ (for the usual topology) such that the exact sequence
$$0\ra H_{\Z} \ra H'_{\Z} \ra \Z \ra 0$$
splits over $U_i$. Splittings correspond to sections $\sigma_i\in H^0(U_i, H'_{\Z})$ mapping to $1$ in $\Z$. The cohomology class $\delta(1)$ is represented by the cocycle $\sigma_i-\sigma_j$.

For each $i$ and each $s\in U_i$, let $\tau_i(s)$ be the image in $H_{s,\C}/F^0H_{s,\C}$ of an element $\alpha_s\in H_{s,\C}$ such that $\sigma_i(s)+\alpha_s\in F^0 H'_{s,\C}$. The Hodge class of the normal function $\nu$ is represented by the cocycle $\tau_i-\tau_j\in H_{\Z}/(H_{\Z}\cap F^0H_{\Z})=H_{\Z}$. Since $\tau_i-\tau_j=\sigma_i-\sigma_j$ through this identification, this concludes the proof.
\end{proof}

\bigskip

Let $S$ be a quasi-projective variety over a finitely generated subfield $k$ of $\C$, with an extension $\nu_{\mathaccent 19 et}$

$$0\ra \hat{H} \ra \hat{H'} \ra \hat{\Z}\ra 0$$
of $\hat{\Z}$-sheaves over $S$. We get an extension class in $H^1(S_{\C}, \hat{H})$ by pulling back to $S_{\C}$, which we will denote by $[\nu_{\mathaccent 19 et}]$.

\bigskip

In case we start with a smooth projective family $\pi : X\ra S$, together with a flat family of algebraic cycles $Z\hookrightarrow X$ of algebraic cycles of codimension $i$ which are homologically equivalent to zero on the fibers of $\pi$, we get an extension of variations of mixed Hodge structures over $S_{\C}$ corresponding to $\nu_Z$, and an extension $\nu_{\mathaccent 19 et}$ of $\hat{\Z}$-sheaves over $S$ induced by the \'etale Abel-Jacobi map. The pull-back of the latter to $S_{\C}$ is the extension of local systems induced by $\nu$. As a consequence, using Artin's comparison theorem between \'etale and singular cohomology in \cite{SGA4}, exp. XI, we get the following ``easy'' part of the comparison theorems between Abel-Jacobi maps.

\begin{prop}\label{comph}
The class $[\nu_{\mathaccent 19 et}]$ is the image in $H^1(S_{\C}, R^{2i-1}\pi_*\hat{\Z}(i))$ of the Hodge class $[\nu]\in H^1(S_{\C}, R^{2i-1}\pi_*\Z(i))$. As a consequence, $[\nu]=0$ if and only if $[\nu_{\mathaccent 19 et}]=0$.
\end{prop}

\paragraph{Remark 1.}
There are of course different ways of computing the value of $[\nu_{\mathaccent 19 et}]$. Indeed, Leray spectral sequences still exist in continuous \'etale cohomology, working in the category of $l$-adic sheaves as defined by Ekedahl in \cite{Ek}. The cycle class of $Z$ induces from the Leray spectral sequence attached to the morphism $\pi$ an element in $H^1_{\mathaccent 19 et}(S, R^{2i-1}\pi_*\hat{\Z}(i))$. This cohomology class maps to an element in $H^1_{\mathaccent 19 et}(S_{\C}, R^{2i-1}\pi_*\hat{\Z}(i))=H^1(S_{\C}, R^{2i-1}\pi_*\hat{\Z}(i))$. Now this class is equal to $[\nu]$. This can either be proved directly or using proposition \ref{comph} and applying the corresponding well-known result for Griffith's Abel-Jacobi map (see \cite{Voib}, lemma 20.20).

\bigskip
\paragraph{Remark 2.}
Proposition \ref{comph} implies the fact that for cycles \emph{algebraically equivalent to zero}, the vanishing of Griffith's Abel-Jacobi invariant is equivalent to the vanishing of the \'etale Abel-Jacobi invariant. This result is well-known, and Raskind gives a proof for zero-cycles in \cite{Rask}, but it does not appear to have been stated explicitly in the literature. We can easily reduce it to the case of divisors on curves, which is treated by Raskind, by the following functoriality argument.

We work over a finitely generated subfield $k$ of $\C$. Let $Z\hookrightarrow X$ be a flat family of cycles of codimension $i$ over a smooth curve $C$ such that the fiber of $Z$ over a geometric point $0$ of $C$ is zero. Changing base to $\C$, the normal function $\nu_Z : C\ra J^i(X)$ takes value in the algebraic part $J^i_{alg}(X)$ of the intermediate Jacobian of $X$. Doing Kummer theory on $J^i_{alg}(X)$, we get a class in $H^1(S, H^{2i-1}(X, \hat{\Z}(i)))$ \footnote{Actually, we would need to change the base field to a field of definition of $J^i_{alg}(X)$ to do so, which we are allowed to do by proposition \ref{inv}, but $J^i_{alg}(X)$ is actually defined over $k$.}. This corresponds to an extension $\nu'_{\mathaccent 19 et}$ of sheaves on $S$ which we claim is $\nu_{\mathaccent 19 et}$, coming from the \'etale Abel-Jacobi map applied to the $Z_s$. The comparison result we need comes from this claim and the Mordell-Weil theorem.

Now since the extension $\nu'_{\mathaccent 19 et}$ obviously splits at the point $0$, we just have to prove that $[\nu'_{\mathaccent 19 et}]=[\nu_{\mathaccent 19 et}]$. An easy functoriality argument reduces this to the case when $X$ is the curve $C$ itself, which concludes. We could also have used functoriality for $\nu_{\mathaccent 19 et}$ itself, but this is a little more cumbersome and is not necessary.

\section{Proof of the theorems}

\subsection{Zero loci for large monodromy groups}

This section is devoted to showing how assuming the family $X\ra S$ has a large monodromy can help study the vanishing locus of the Abel-Jacobi map and deduce theorem \ref{gros} and theorem \ref{df}. This kind of argument is very much inspired by \cite{Voabs}, where it appears in the pure case as a criterion for Hodge classes to be absolute.

The main idea is the following : start with an extension $0\ra H' \ra H \ra \Z \ra 0$ of variations of mixed Hodge structures on a quasi-projective variety $S$. If the monodromy representation on $H'$ has no nontrivial invariant part, then any global section of $H$ is in $F^0 H$, the filtration $F$ being the Hodge filtration. This remark allows us to reduce the question of the splitting of the previous exact sequence to a geometric question, and allows for comparison theorems.

In the setting of normal functions, this is equivalent to the following, which has been observed a long time ago. Under this assumption, the normal function with value in the $i$-th intermediate Jacobian is determined by its Hodge class, see \cite{Grif2}. This has been used for instance by Manin in the proof of Mordell's conjecture over function fields in \cite{Man}. Our argument does not proceed along these lines for convenience, but part of it could easily be translated using Griffiths' results and the Leray spectral sequence.

\bigskip

Recall the notations of the introduction. We have a smooth projective family over a quasi-projective base $\pi : X\ra S$, together with a flat family of algebraic cycles $Z\ra X$ of pure codimension $i$. Everything is defined over a finitely generated field $k$ of characteristic zero. As far as our results are concerned, and taking into account proposition \ref{inv}, standard spreading techniques allow us to assume without loss of generality that $k$ is a number field. Suppose that for any geometric point $s$ of $S$, $Z_s$ is homologically equivalent to zero in $X_s$.  Fix an embedding of $k$ in $\C$. We get the normal function $\nu_Z$ on $S(\C)$, which is a holomorphic section of the bundle of intermediate Jacobians over $S(\C)$.

We have the following exact sequence of local systems of $\hat{\Z}$-sheaves on $S$, canonically attached to the family $Z$ of algebraic cycles
\begin{equation}\label{varl}
0\ra R^{2i-1}\pi_* \hat{\Z}(i) \ra H_{\mathaccent 19 et} \ra \hat{\Z} \ra 0.
\end{equation}
Since $\hat{\Z}$ is flat over $\Z$, the pull-back to $S_{\C}$ of this sequence of sheaves is the tensor product by $\hat{\Z}$ of the exact sequence
\begin{equation}\label{varh}
0\ra R^{2i-1}(\pi_{\C})_* \Z(i) \ra H \ra \Z \ra 0
\end{equation}
of local systems used to compute Griffiths' Abel-Jacobi map. Those local systems are underlying to variations of mixed Hodge structures. Saying that $\nu_Z$ vanishes on $S_{\C}$ is equivalent to saying that $S$ is equal to the locus of Hodge classes of $H$ which map to $1$ in $\Z$.

We will deduce our theorems from the following result.

\begin{thm}
In the above setting, assume that the locally constant sheaf $R^{2i-1}(\pi_{\C})_* \C$ has no nonzero global section over $S_{\C}$. Then the following are equivalent :
\begin{itemize}
\item[(i)] The normal function $\nu$ associated to $Z_{\C}$ vanishes on $S_{\C}$.
\item[(ii)] For every closed point $s$ of $S$ with value in a finitely generated field $K$, the image of $Z_s$ by the \'etale Abel-Jacobi map vanishes in the group $H^1(G_K, H^{2i-1}(\overline{X_s}, \widehat{\Z}(i)))$.
\item[(iii)] For any automorphism $\sigma$ of $\C$, the normal function $\nu^{\sigma}$ associated to $Z^{\sigma}=Z_{\C}\times_{\sigma} \mathrm{Spec}(\C)$ vanishes on $S^{\sigma}$.
\end{itemize}
\end{thm}

\bigskip

\begin{proof}[P\textbf{roof of $(i)\Rightarrow (ii)$}]
Fix a point $s$ of $S$ with value in a finitely generated field $L$, and let $\overline{s}$ be a complex point of $S$ lying over $s$. Under our hypothesis, we have an injective map $(H_{\mathaccent 19 et})_{\overline{s}}^{\pi_1^{\mathaccent 19 et}(S_{\C}, \overline{s})}\ra \hat{\Z}$. This is actually an isomorphism. Indeed, Baire's theorem applied to the locus of Hodge classes of $H$ in $S_{\C}$ mapping to $1$ in $\Z$ shows that in order for $S_{\C}$ to be equal to this locus, which is a countable union of analytic subvarieties, there has to be a nonzero global section of $H$ which is a Hodge class in every fiber of $H$ -- and maps to $1$ in $\Z$. The image in $H_{\overline{s}}\otimes\hat{\Z}=(H_{\mathaccent 19 et})_{\overline{s}}$ of this section lies in $(H_{\mathaccent 19 et})_{\overline{s}}^{\pi_1^{\mathaccent 19 et}(S_{\C}, \overline{s})}$ and maps to $1$ in $\hat{\Z}$.

Now let $G_L$ be the absolute Galois group of $L$. We have an exact sequence $1\ra \pi_1^{\mathaccent 19 et}(S_{\C}, \overline{s}) \ra \pi_1^{\mathaccent 19 et}(S\times \mathrm{Spec}(L), \overline{s}) \ra G_L\ra 1$, together with a splitting of this exact sequence. The full algebraic fundamental group acts on $(H_l)_{\overline{s}}$, and the map $(H_{\mathaccent 19 et})_{\overline{s}}\ra \hat{\Z}$ is equivariant with respect to the trivial action on $\hat{\Z}$. It follows that the group $G_L$ acts trivially on $(H_{\mathaccent 19 et})_{\overline{s}}^{\pi_1^{\mathaccent 19 et}(S_{\C}, \overline{s})}\xrightarrow{\sim} \hat{\Z}$. This proves that the \'etale Abel-Jacobi invariant of $Z_s$ is zero.
\end{proof}

\bigskip

\begin{proof}[P\textbf{roof of $(ii)\Rightarrow (iii)$}]

It is enough to prove the case where $\sigma$ is the identity. Fix a prime number $l$, and denote by $H_l$ the $l$-adic part of $H_{\mathaccent 19 et}$ Let $\overline{s}$ be a geometric point of $S$. Using the same notation as in the previous proof, the algebraic fundamental group $\pi_1^{\mathaccent 19 et}(S, \overline{s})$ acts on $(H_{\mathaccent 19 et})_{\overline{s}}$. For any point $s'$ of $S$ with value in a field $L$, the absolute Galois group $G_L$ of $L$ maps into $\pi_1^{\mathaccent 19 et}(S, \overline{s})$. According to a lemma by Terasoma appearing in \cite{Ter}, theorem 2, there exists such an $L$-valued point $s'$, with $L$ a number field, such that $G_L$ and $\pi_1^{\mathaccent 19 et}(S, \overline{s})$ have the same image in the linear group GL$((H_{l})_{\overline{s}})$. Since by assumption $G_L$ fixes an element mapping to $1\in \Z_l$, we get an element of $(H_l)_{\overline{s}}$, mapping to $1\in \Z_l$, which is fixed by the whole monodromy group. In other words, the $l$-adic part of the exact sequence (\ref{varl}) splits over $S$. This being true for any prime number $l$, the exact sequence (\ref{varl}) splits over $S$.

 This means that the local system $H_{\mathaccent 19 et}$ on $S$ has a nonzero global section. As a consequence, $H^0(S_{\C}, H\otimes\Q)\neq 0$, and as before we get an isomorphism $H^0(S_{\C}, H\otimes\Q)\simeq \Q$ as local systems, the map being induced by the morphism $H\otimes\Q\ra \Q$ of variations of mixed Hodge structures over $S_{\C}$. It is a result of Steenbrink and Zucker in \cite{SZ1}, th. 4.1\footnote{This theorem is a generalization of Deligne's global invariant cycles theorem, which is a fundamental tool of \cite{Voabs}.} that the $H^0(S_{\C}, H\otimes\Q)$ carries a canonical mixed Hodge structure which makes it a constant subvariation of mixed Hodge structures of $H$\footnote{The result of Steenbrink and Zucker is stated for variations of mixed Hodge structures of geometric origin -- which is our case -- over a one-dimensional base. The fact that $H^0(S_{\C}, H\otimes\Q)$ carries a canonical Hodge structure for $S$ of any dimension is straightforward by restricting to a curve which is an intersection of hyperplane sections and using Lefschetz' hyperplane theorem.}. The isomorphism of $H^0(S_{\C}, H\otimes\Q)$ with $\Q$ is a morphism of mixed Hodge structures, which proves $H^0(S_{\C}, H\otimes\Q)$ consists of Hodge classes.

 This shows that the exact sequence (\ref{varh}) of variations of mixed Hodge structures splits rationally. We want to prove that it splits over $\Z$. We just proved that a splitting of the subjacent extension of local systems over $S_{\C}$ gives a splitting of (\ref{varh}), so we just have to prove that the exact sequence of local systems splits.

 Let $\alpha\in H^0(S_{\C}, H\otimes\Q)$ be the class mapping to $1\in \Q$. For any prime number $l$, the image of $\alpha\in H^0(S_{\C}, H\otimes\Q)$ is the only class mapping to $1\in \Q_l$, which shows that this image belongs to $H^0(S_{\C}, H\otimes\Z_l)$, since the exact sequence (\ref{varl}) is split over $S_{\C}$. The only way for this to be true is that $\alpha$ belongs to $H^0(S_{\C}, H)$, which precisely means that the exact sequence we are considering splits.
\end{proof}

\bigskip

\begin{proof}[P\textbf{roof of $(iii)\Rightarrow (i)$}]
This is obvious.
\end{proof}

\bigskip

Let us now use the notations of the introduction. The equivalence $(i)\Leftrightarrow (ii)$ we just proved readily implies theorem \ref{gros} by restriction to the component $T$ of the zero locus of $\nu_Z$, which is assumed to be algebraic, under the assumption that the local system $R^{2i-1}(\pi_{\C})_* \C$ has no nonzero global section.

\bigskip

\begin{proof}[P\textbf{roof of theorem \ref{df}}]
Let $\overline{k}$ be an algebraic closure of $k$, and let $T'$ be the Zariski-closure of $T(\C)$ in the $k$-scheme $S$. The previous theorem shows that the orbit of $T(\C)$ in $S$ under the action of the Galois group Aut$(\C/k)$ is included in the zero locus of $\nu_Z$. Since this orbit is dense in $T'(\C)$ for the usual topology, it follows that $\nu_Z$ vanishes on $T'(\C)$. By assumption, $T$ is an irreducible component of the zero locus of $\nu_Z$. It follows that $T$ is an irreducible component of the algebraic variety $T'$ defined over $k$, which proves that $T$ is defined over a finite extension of $k$.

This shows that for any automorphism $\sigma$ of $\C$ fixing $k$, the set $\sigma(T(\C))$ is included in the zero locus of $\nu_Z$. Now consider the subscheme $T^{\sigma}$ of $S$, which has $\sigma(T(\C))$ as set of complex points. We just showed that its reduced subscheme is included in the zero locus of $\nu_Z$ as an analytic subvariety, and it is irreducible. Let $V$ be the irreducible component of the zero locus of $\nu_Z$ containing $\sigma(T(\C))$. We want to show that $V=T^{\sigma}$ as analytic varieties. Let $n$ be a nonnegative integer. We can consider the artinian local rings $\mathcal{O}_{V, \sigma(t)}/\mathfrak{m}_{\sigma(t)}^n$ and $\mathcal{O}_{T^{\sigma}, \sigma(t)}/\mathfrak{m}_{\sigma(t)}^n$, $\mathfrak{m}_{\sigma(t)}^n$ denoting indifferently the maximal ideals of both local rings.

The rings $\mathcal{O}_{T^{\sigma}, \sigma(t)}/\mathfrak{m}_{\sigma(t)}^n$ and $\mathcal{O}_{T, t}/\mathfrak{m}_{t}^n$ are canonically isomorphic, because the schemes $T$ and $T^{\sigma}$ are. On the other hand, we can explicitly describe $\mathcal{O}_{T, t}/\mathfrak{m}_{t}^n$ and $\mathcal{O}_{V, \sigma(t)}/\mathfrak{m}_{\sigma(t)}^n$, as loci of Hodge classes, using the Gauss-Manin connection on $\mathcal{H}$ and Griffiths transversality. This is explained in \cite{Voib} in the case of pure Hodge structures, and explicitly stated for $n=1$, see lemma 17.16. Our case follows \emph{mutatis mutandis}. As a consequence, since the Gauss-Manin connection is algebraic, see \cite{KO}, we have a canonical isomorphism between $\mathcal{O}_{T, t}/\mathfrak{m}_{t}^n$ and $\mathcal{O}_{V, \sigma(t)}/\mathfrak{m}_{\sigma(t)}^n$.

This discussion shows that $\mathcal{O}_{T^{\sigma}, \sigma(t)}/\mathfrak{m}_{\sigma(t)}^n$ and $\mathcal{O}_{V, \sigma(t)}/\mathfrak{m}_{\sigma(t)}^n$ are isomorphic as subrings of $\mathcal{O}_{S, \sigma(t)}/\mathfrak{m}_{\sigma(t)}^n$. Since this holds for all $n$, and since the reduced subscheme of $T^{\sigma}$ is included in $V$, we get an equality $V=T^{\sigma}$, which shows that $T^{\sigma}$ is an irreducible component of the zero locus of $\nu_Z$.
\end{proof}

\subsection{Application}

As in \cite{Voabs}, there are many situations where one can easily check that the conditions of theorems \ref{gros} and \ref{df}. Let us give one example.

\begin{thm}\label{CY}
Let $\pi : X\ra S$ be a smooth projective family of complex Calabi-Yau threefolds over a quasiprojective base such that the induced map from $S$ to the corresponding moduli space is finite, and let $Z\hookrightarrow X$ be a flat family of curves in $X$ which are homologous to $0$ in the fibers of $\pi$. Assume everything is defined over a finitely generated field $k$. Let $\nu$ be the associated normal function.
\begin{itemize}
 \item[(i)] Let $T$ be an irreducible component of the zero locus of $\nu$ which is algebraic. Assume that $T$ is of positive dimension and that for a general complex point $t$ of $T$, the intermediate jacobian $J^2(X_t)$ has no abelian factor. Then $T$ is defined over a finite extension of $k$, all its conjugate are irreducible components of the zero locus of $\nu$, and for every closed point $t$ of $T$ with value in a finitely generated field, the \'etale Abel-Jacobi invariant of $Z_t$ is zero.
 \item[(ii)] Let $T$ be a subvariety of $S$ of positive dimension defined over a finitely generated field. Assume that for a general complex point $t$ of $T$, the intermediate jacobian $J^2(X_t)$ has no abelian factor and that for any point $t$ of $T$ with value in a finitely generated field, the \'etale Abel-Jacobi invariant of $Z_t$ is zero. Then $\nu$ vanishes on $T$.
\end{itemize}
\end{thm}

\begin{proof}
 In order to apply our preceding results, we only have to check that in both situations above, the local system $R^3 \pi_* \Z$ has no global section over $T(\C)$. First of all, since the restriction of $\pi$ to $T$ is a nontrivial family of Calabi-Yau threefolds, the Hodge structure on $H^0(T(\C), R^3 \pi_* \Z)$ is of type $\{(2,1), (1,2)\}$. Indeed, the infinitesimal Torelli theorem for Calabi-Yau varieties, see \cite{Voib}, th. 10.27, shows that the fixed part of $R^3 \pi_* \Z$ cannot have a part of type $(3,0)$. Now this proves that the invariant part of $R^3 \pi_* \Z$ corresponds to a constant abelian subvariety of the Jacobian fibration $J^2(X_T/T)$, which has to be zero by assumption. This shows that the local system $R^3 \pi_* \Z$ has no global section.
\end{proof}

\subsection{Proof of theorem \ref{inter}}

Let us now prove theorem \ref{inter}. The first part uses Terasoma's lemma and classical results about normal functions and their infinitesimal behavior.

\begin{proof}[P\textbf{roof of $(i)$}]
In the situation of the theorem, we can use Terasoma's lemma as before to see that the exact sequence $\nu_{\mathaccent 19 et}$ of $\hat{\Z}$-sheaves on $S$ associated to $Z\hookrightarrow X$ is split, which implies that $[\nu_{\mathaccent 19 et}]$ is zero, and shows that the Hodge class $[\nu_Z]$ of $\nu_Z$ is zero by proposition \ref{comph}.

According to fundamental results of Griffiths, see \cite{Grif2}, a normal function with zero Hodge class is constant in the fixed part of the intermediate Jacobian. In our case, since $\nu_Z$ vanishes at some complex point of $S$, this shows that $\nu_Z=0$.
\end{proof}

\bigskip

The proof of the next theorem, which is the remaining part of theorem \ref{inter}, is quite different in spirit from what we just did, as it includes deeper number-theoretic arguments. In light of the last proof, one should consider it as an analogue of the important fact that the Hodge class of a normal function determines it up to a constant in the invariant part of the intermediate Jacobian. We feel that it might be possible to use such ideas to prove stronger results on the \'etale Abel-Jacobi map.

\begin{thm}
Let $T$ be an irreducible component of the zero locus of $\nu_Z$. Assume that $T$ is algebraic and let $k$ be a finitely generated field over which $T$ is defined. Assume that there exists a point $t$ of $T$ with value in a finitely generated field such that $aj_{\mathaccent 19 et}(Z_t)$ is zero. Then for every point of $T$ with value in a finitely generated field, $aj_{\mathaccent 19 et}(Z_t)$ is zero.
\end{thm}

\begin{proof}
 Any specialization of $t$ satisfies the assumption of the theorem, so we can assume that $t$ is a closed point. Up to a base change and using proposition \ref{inv}, we can assume that $T$ has value in the base field $k$. We need to prove that the exact sequence (\ref{varl}) splits over $T$. It suffices to prove that its $l$-adic part splits over $T$ for any prime number $l$. Fix a prime number $l$ and denote once again by $H_l$ the $l$-adic part of $H_{\mathaccent 19 et}$. Let $\overline{t}$ be a geometric point of $T$ lying over $t$. We have an exact sequence
$$1\ra \pi_1^{\mathaccent 19 et}(T_{\C}, \overline{t}) \ra \pi_1^{\mathaccent 19 et}(T, \overline{t}) \ra G_k \ra 1.$$
The last arrow admits a section $\sigma$ coming from the rational point $t$. Now by assumption, the following sequence is exact.
\begin{equation}\label{fix}
 0\ra (H^{2i-1}(X_{\overline{t}}, \Z_l(i)))^{\pi_1^{\mathaccent 19 et}(T_{\C}, \overline{t})} \ra (H_{l, \overline{t}})^{\pi_1^{\mathaccent 19 et}(T_{\C}, \overline{t})} \ra \Z_l \ra 0.
\end{equation}
Indeed, the vanishing of $\nu_Z$ implies that there exists a global section of $H$ over $T_{\C}$ mapping to $1$ in $\Z$, which implies the surjectivity of the last arrow.

The Galois group $G_k$ acts on $(H_{l, \overline{t}})^{\pi_1^{\mathaccent 19 et}(T_{\C}, \overline{t})}$, either through $\sigma$ or through the previous exact sequence -- those are the same actions. For every finite place $\mathfrak{p}$ of $k$, fix a Frobenius element $F_{\mathfrak{p}}$ in a decomposition group of $\mathfrak{p}$. It follows from \cite{Se}, prop. 6, 8 and 12, and \cite{Se2}, th. 1, that the extension of $G_k$-modules (\ref{fix}) is split if and only if for almost every finite place $\mathfrak{p}$ of $k$, there exists an element $h$ of $(H_{l, \overline{t}})^{\pi_1^{\mathaccent 19 et}(T_{\C}, \overline{t})}$, mapping to $1$ in $\Z_l$, such that $F_{\mathfrak{p}}(h)=h$. In the first paper, using \v{C}ebotarev's theorem, Serre indeed proves that if (\ref{fix}) splits at almost every finite place, then it comes from an extension of $G$-modules, where $G$ is the image of $G_k$ in GL$( (H^{2i-1}(X_{\overline{t}}, \Z_l(i)))^{\pi_1^{\mathaccent 19 et}(T_{\C}, \overline{t})})$. The second paper proves the splitting using the Lie algebra of $G$.

Assume $\mathfrak{p}$ does not divide $l$ and that $X_t$ has good reduction at $\mathfrak{p}$. Since the exact sequence of $G_k$-modules
$$0\ra H^{2i-1}(X_{\overline{t}}, \Z_l(i)) \ra H_{l, \overline{t}} \ra \Z_l \ra 0$$
is split by assumption, there exists $h'\in H_{l, \overline{t}}$, mapping to $1$ in $\Z_l$, such that $F_{\mathfrak{p}}(h')=h'$. On the other hand, since $F_{\mathfrak{p}}$ acts trivially on $\Z_l$ and has weight $-1$ on $(H^{2i-1}(X_{\overline{t}}, \Z_l(i)))^{\pi_1^{\mathaccent 19 et}(T_{\C}, \overline{t})}$ by the Weil conjectures, there exists $h\in (H_{l, \overline{t}})^{\pi_1^{\mathaccent 19 et}(T_{\C}, \overline{t})}\otimes\Q$, mapping to $1$ in $\Z_l$, such that $F_{\mathfrak{p}}(h)=h$. Since $H^{2i-1}(X_{\overline{t}}, \Z_l(i))^{F_{\mathfrak{p}}}=0$, we have $h'=h$, which shows that $h$ lies in $(H_{l, \overline{t}})^{\pi_1^{\mathaccent 19 et}(T_{\C}, \overline{t})}$. This proves that the exact sequence (\ref{fix}) splits, which implies that the $l$-adic part of the exact sequence (\ref{varl}) is split over $T$ for any prime number $l$, and concludes the proof.
\end{proof}

{\par\vskip.3truein\relax
\leftline{\hskip2truein\relax Fran\c{c}ois Charles}
\leftline{\hskip2truein\relax \'Ecole Normale Sup\'erieure}
\leftline{\hskip2truein\relax 45 rue d'Ulm}
\leftline{\hskip2truein\relax 75\thinspace 230 PARIS CEDEX 05}
\leftline{\hskip2truein\relax FRANCE}
\leftline{\hskip2truein\relax\tt francois.charles@ens.fr}
\relax}

\end{document}